\documentclass[a4, 12pt]{amsart}
\usepackage{amssymb}
\usepackage{amstext}
\usepackage{amsmath}
\usepackage{amscd}
\usepackage{latexsym}
\usepackage{amsfonts}
\usepackage{color}
\usepackage{enumerate}
\usepackage{textcomp}
\usepackage[all]{xy}

\usepackage{appendix}

\theoremstyle{plain}
\newtheorem{thm}{Theorem}[section]
\newtheorem*{thm*}{Theorem}
\newtheorem*{cor*}{Corollary}
\newtheorem*{defn*}{Definition}

\newtheorem{prop}[thm]{Proposition}
\newtheorem{lem}[thm]{Lemma}
\newtheorem{cor}[thm]{Corollary}
\newtheorem{claim}[thm]{Claim}
\newtheorem*{claim*}{Claim}
\newtheorem*{ac}{Acknowledgments}

\theoremstyle{definition}
\newtheorem{defn}[thm]{Definition}
\newtheorem{ex}[thm]{Example}
\newtheorem{rem}[thm]{Remark}
\newtheorem{fact}[thm]{Fact}

\theoremstyle{remark}

\numberwithin{equation}{thm}

\def\Coker{\mathrm{Coker}}
\def\Ker{\mathrm{Ker}}
\def\Im{\mathrm{Im}}

\def\a{\mathfrak a}

\def\m{\mathfrak m}

\def\p{\mathfrak p}
\def\q{\mathfrak q}

\def\M{\mathfrak M}

\def\H{\mathrm{H}}

\newcommand{\rma}{\mathrm{a}}

\newcommand{\calC}{\mathcal{C}}
\newcommand{\calD}{\mathcal{D}}
\newcommand{\calE}{\mathcal{E}}
\newcommand{\calF}{\mathcal{F}}
\newcommand{\calG}{\mathcal{G}}

\newcommand{\calI}{\mathcal{I}}

\newcommand{\calM}{\mathcal{M}}

\newcommand{\calR}{\mathcal{R}}
\newcommand{\calS}{\mathcal{S}}

\def\depth{\mathrm{depth}}
\def\Supp{\mathrm{Supp}}

\def\Ass{\mathrm{Ass}}
\def\Assh{\mathrm{Assh}}
\def\Min{\mathrm{Min}}

\def\Spec{\mathrm{Spec}}

\begin{document}

\setlength{\baselineskip}{12pt}
\title{Sequentially Cohen-Macaulay Rees algebras}
\author{Naoki Taniguchi}
\address{Department of Mathematics, School of Science and Technology, Meiji University, 1-1-1 Higashi-mita, Tama-ku, Kawasaki 214-8571, Japan}
\email{taniguti@math.meiji.ac.jp}
\urladdr{http://www.isc.meiji.ac.jp/~taniguci/}

\author{Tran Thi Phuong}
\address{Department of Mathematics, School of Science and Technology, Meiji University, 1-1-1 Higashi-mita, Tama-ku, Kawasaki 214-8571, Japan}
\email{sugarphuong@gmail.com}

\author{Nguyen Thi Dung}
\address{Thai Nguyen University of Agriculture and Forestry, Thai Nguyen, Vietnam}
\email{xsdung050764@gmail.com}

\author{Tran Nguyen An}
\address{Thai Nguyen University of Pedagogical, Thai Nguyen, Vietnam}
\email{antrannguyen@gmail.com}


\thanks{The first author was partially supported by Grant-in-Aid for JSPS Fellows 26-126 and by JSPS Research Fellow. The second author was partially supported by JSPS KAKENHI 26400054. The third and the fourth author were partially supported by a Grant of Vietnam Institute for Advanced Study in Mathematics (VIASM) and Vietnam National Foundation for Science and Technology Development (NAFOSTED). }

\begin{abstract}
This paper studies the question of when the Rees algebras associated to arbitrary filtration of ideals are sequentially Cohen-Macaulay. Although this problem has been already investigated by \cite{CGT}, their situation is quite a bit of restricted, so we are eager to try the generalization of their results. 
\end{abstract}

\maketitle
\tableofcontents
\section{Introduction}

The notion of sequentially Cohen-Macaulay property was originally introduced by R. P. Stanley (\cite{St}) for Stanley-Reisner algebras and then it has been furiously explored by many researchers, say D. T. Cuong, N. T. Cuong, S. Goto, P. Schenzel and others (see \cite{CC, CGT, GHS, Sch}), from the view point of not only combinatorics, but also commutative algebra. The purpose of this paper is to investigate the question of when the Rees algebras are sequentially Cohen-Macaulay, which has a previous research by \cite{CGT}. In \cite{CGT} they gave a characterization of the sequentially Cohen-Macaulay Rees algebras of $\m$-primary ideals (\cite[Theorem 5.2, Theorem 5.3]{CGT}). However their situation is not entirely satisfactory, so we are eager to analyze the case where the ideal is not necessarily $\m$-primary. More generally we want to deal with the sequentially Cohen-Macaulayness of the Rees modules since the sequentially Cohen-Macaulay property is defined for any finite modules over a Noetherian ring. Thus the main problem of this paper is when the Rees modules associated to arbitrary filtration of modules are sequentially Cohen-Macaulay.

Let $R$ be a commutative Noetherian ring, $M \neq (0)$ a finitely generated $R$-module with $d = \dim_R M < \infty$. 
Then we consider a filtration
$$
\calD : D_0 :=(0) \subsetneq D_1 \subsetneq  D_2 \subsetneq  \ldots \subsetneq D_{\ell} = M
$$
of $R$-submodules of $M$, which we call {\it the dimension filtration of $M$}, if $D_{i-1}$ is the largest $R$-submodule of $D_i$ with $\dim_R D_{i-1}< \dim_R D_i$ for $1\le i \le \ell$, here $\dim_R(0) = -\infty$ for convention. We note here that our notion of dimension filtration is based on \cite{GHS} and slightly different from that of the original one given by P. Schenzel (\cite{Sch}), however let us adopt the above definition throughout this paper. Then we say that $M$ is  {\it a sequentially Cohen-Macaulay $R$-module}, if the quotient module $C_i = D_i/D_{i-1}$ of $D_i$ is a Cohen-Macaulay $R$-module for every $1 \le i \le \ell$. In particular, a Noetherian ring  $R$ is called {\it a sequentially Cohen-Macaulay ring}, if $\dim R < \infty$ and $R$ is a sequentially Cohen-Macaulay module over itself. 

Let us now state our results, explaining how this paper is organized. In Section 2 we sum up the notions of the sequentially Cohen-Macaulay properties and filtrations of ideals and modules.  
In Section 3 we shall give the proofs of the main results of this paper, which are stated as follows.

Suppose that $R$ is a local ring with maximal ideal $\m$. Let $\calF = \{F_n\}_{n\in \mathbb{Z}}$ be a filtration of ideals of $R$ such that $F_1\neq R$, $\mathcal M = \{M_n\}_{n\in \mathbb{Z}}$ a $\calF$-filtration of $R$-submodules of $M$. Then we put 
$$
\mathcal R =\sum_{n\geq0} F_nt^n  \subseteq R[t], \ \ \mathcal R^{\prime}  =\sum_{n\in\mathbb Z} F_nt^n  \subseteq R[t,t^{-1}], \ \
\mathcal G =\mathcal R^{\prime}/t^{-1}\mathcal R^{\prime} 
$$
and call them {\it the Rees algebra, the extended Rees algebra and the associated graded ring of $\calF$}, respectively. Similarly we set
$$
\mathcal R ( \calM) =\sum_{n\geq0} t^n \otimes M_n \subseteq  R[t]\otimes_R M, \ \
\mathcal R^{\prime} (\calM) = \sum_{n\in\mathbb Z}  t^n \otimes M_n \subseteq R[t,t^{-1}]\otimes_R M
$$
and
$$
\mathcal G(\calM) =\mathcal R^{\prime}(\calM)/t^{-1}\mathcal R^{\prime}(\calM) 
$$
which we call {\it the Rees module, the extended Rees module and the associated graded module of $\calM$}, respectively. Here $t$ stands for an indeterminate over $R$. 
We also assume that $\calR $ is a Noetherian ring and $\calR(\calM)$ is a finitely generated $\calR$-module. Set 
$$
\calD_i = \{M_n\cap D_i\}_{n\in \mathbb Z}, \ \ 
\calC_i = \{[(M_n\cap D_i)+D_{i-1}]/D_{i-1}\}_{n\in \mathbb Z}.
$$
for every $1 \le i \le \ell$. Then $\mathcal D_i$ (resp. $\mathcal C_i$) is a $\calF$-filtration of $R$-submodules of $D_i$ (resp. $C_i$).
With this notation the main results of this paper are the following, which are the natural generalization of the results \cite[Theorem 5.2, Theorem 5.3]{CGT}.

\begin{thm}\label{1.2} The following conditions are equivalent.
\begin{enumerate}[{\rm (1)}]
\item $\mathcal R^{\prime} (\calM)$ is a sequentially Cohen-Macaulay $\mathcal R^{\prime} $-module. 
\item $\mathcal G(\calM)$ is a sequentially Cohen-Macaulay $\calG$-module and $\{\mathcal G (\calD_i)\}_{0 \le i \le \ell}$ is the dimension filtration of $\calG(\calM)$.
\end{enumerate}
When this is the case, $M$ is a sequentially Cohen-Macaulay $R$-module.
\end{thm}

Let $\M$ be a unique graded maximal ideal of $\calR$. We set 
$$
\rma(N) = \max \{n\in \Bbb Z \mid [\H^t_\M(N)]_n \neq (0) \}
$$
for a finitely generated graded $\calR$-module $N$ of dimension $t$, and call it {\it the a-invariant of $N$} (see \cite[DEFINITION (3.1.4)]{GW}). Here $\{[\H^t_\M(N)]_n\}_{n \in \Bbb Z}$ stands for the homogeneous components of the $t$-th graded local cohomology module $\H^t_\M(N)$ of $N$ with respect to $\M$.  

\begin{thm}\label{1.3} Suppose that $M$ is a sequentially Cohen-Macaulay $R$-module and $F_1 \nsubseteq \p$ for every $\p \in \Ass_R M$. Then the following conditions are equivalent.
\begin{enumerate}[{\rm (1)}]
\item $\calR (\calM)$ is a sequentially Cohen-Macaulay $\calR$-module. 
\item $\calG(\calM)$ is a sequentially Cohen-Macaulay $\calG$-module, $\{\calG (\calD_i)\}_{0 \le i \le \ell}$ is the dimension filtration of $\calG(\calM)$  and  $\rma(\calG (\calC_i))<0$ for every $1 \le i \le \ell$.
\end{enumerate}
When this is the case, $\calR'(\calM)$ is a sequentially Cohen-Macaulay $\calR^{\prime}$-module.
\end{thm}

In Section 4 we focus our attention on the case of graded rings. In the last section we will explore the application of Theorem \ref{6.7} to the Stanley-Reisner algebras of shellable complexes (Theorem \ref{7.4}).


\section{Preliminaries}
In this section we summarize some basic results on sequentially Cohen-Macaulay properties and filtration of ideals and modules, which we will use throughout this paper. 

Let $R$ be a Noetherian ring, $M \ne (0)$ a finitely generated $R$-module of dimension $d$. We put
$$
\Assh_R M = \{\p \in \Supp_R M \mid \dim R/\p = d\}.
$$
For each $n \in \Bbb Z$, there exists the largest $R$-submodule $M_n$ of $M$ with $\dim_R M_n \le n$.
Let 
\begin{eqnarray*}
\calS (M) &=& \{\dim_R N \mid N \ \text{is~an~}R\text{-submodule~of~}M, N \ne (0)\} \\
                 &=& \{\dim R/\p \mid \p \in \Ass_RM \}.
\end{eqnarray*}
We set $\ell = \sharp \calS (M)$ and write $\calS (M) = \{d_1< d_2 < \cdots < d_\ell = d\}$. 
Let $D_i = M_{d_i}$ for each $1 \le i \le \ell$. 
We then have a filtration
$$
\calD : D_0 :=(0) \subsetneq D_1 \subsetneq  D_2 \subsetneq  \ldots \subsetneq D_{\ell} = M
$$
of $R$-submodules of $M$, which we call {\it the dimension filtration of $M$}.
We put $C_i = D_i/D_{i-1}$ for every $1 \le i \le \ell$. 

\begin{defn}[\cite{Sch, St}]\label{3.1}
We say that $M$ is {\it a sequentially Cohen-Macaulay $R$-module}, if $C_i$ is Cohen-Macaulay for every $1 \le i \le \ell$. The ring $R$ is called {\it a sequentially Cohen-Macaulay ring}, if $\dim R < \infty$ and $R$ is a sequentially Cohen-Macaulay module over itself. 
\end{defn}

\noindent The typical examples of sequentially Cohen-Macaulay ring is the Stanley-Reisner algebra $k[\Delta]$ of a shellable complex $\Delta$ over a field $k$. Also every one-dimensional Noetherian local ring is sequentially Cohen-Macaulay. Moreover, if $M$ is a Cohen-Macaulay module over a Noetherian local ring, then $M$ is sequentially Cohen-Macaulay, and the converse holds if $M$ is unmixed.

Firstly let us note the non-zerodivisor characterization of sequentially Cohen-Macaulay modules.

\begin{prop}\label{3.5}
Let $(R, \m)$ be a Noetherian local ring, $M \neq (0)$ a finitely generated $R$-module. Let $x \in \m$ be a non-zerodivisor on $M$. Then the following conditions are equivalent.
\begin{enumerate}[$(1)$]
\item $M$ is a sequentially Cohen-Macaulay $R$-module.
\item $M/xM$ is a sequentially Cohen-Macaulay $R/(x)$-module and $\{D_i/xD_i\}_{0 \le i \le \ell}$ is the dimension filtration of $M/xM$.
\end{enumerate}
\end{prop}

\begin{proof}
Notice that $x \in \m$ is a non-zerodivisor on $C_i$ and $D_i$ for all $1\le i\le \ell$ (See \cite[Corollary 2.3]{Sch}). Therefore we get a filtration
$$
D_0/xD_0 = (0) \subsetneq D_1/xD_1 \subsetneq \cdots \subsetneq D_{\ell}/xD_{\ell} = M/xM
$$
of $R/(x)$-submodules of $M/xM$. Then the assertion is a direct consequence of \cite[Theorem 2.3]{GHS}.
\end{proof}

\noindent The implication $(2) \Rightarrow (1)$ is not true without the condition that $\{D_i/xD_i\}_{0 \le i \le \ell}$ is the dimension filtration of $M/xM$. For instance, let $R$ be a $2$-dimensional Noetherian local domain of depth $1$. Then $R/(x)$ is sequentially Cohen-Macaulay for every $0 \neq x \in R$, but $R$ is not. Besides this, let $I$ be an $\m$-primary ideal in a regular local ring $(R, \m)$ of dimension $2$. Then $I$ is not a sequentially Cohen-Macaulay $R$-module, even though $I/xI$ is, where $0 \neq x \in \m$. 
These examples show that \cite[Theorem 4.7]{Sch} is not true in general.

From now on, we shall quickly review some preliminaries on filtrations of ideals and modules. Let $R$ be a commutative ring, $\calF = \{F_n\}_{n \in \Bbb Z}$ a filtration of ideals of $R$, that is, $F_n$ is an ideal of $R$, $F_n \supseteq F_{n+1}$, $F_m F_n \subseteq F_{m+n}$ for all $m, n \in \Bbb Z$ and $F_0 = R$. Then we put
$$
\calR=\calR(\calF) = \sum_{n \ge 0}F_nt^n \subseteq R[t], \ \ \calR'=\calR'(\calF) = \sum_{n \in \Bbb Z}F_nt^n \subseteq R[t, t^{-1}]
$$
and call them {\it the Rees algebra, the extended Rees algebra of $R$ with respect to $\calF$}, respectively. Here $t$ stands for an indeterminate over $R$.

Let $M$ be an $R$-module, $\calM=\{M_n\}_{n \in \Bbb Z}$ an $\calF$-filtration of $R$-submodules of $M$, that is, $M_n$ is an $R$-submodule of $M$, $M_n \supseteq M_{n+1}$, $F_m M_n \subseteq M_{m+n}$ for all $m, n \in \Bbb Z$ and $M_0 = M$. We set
$$
\calR(\calM)= \sum_{n \ge 0} t^n \otimes M_n \subseteq R[t] \otimes_R M, \ \ \calR'(\calM) = \sum_{n \in \Bbb Z} t^n \otimes M_n \subseteq R[t, t^{-1}] \otimes_R M
$$
which we call  {\it the Rees module, the extended Rees module of $M$ with respect to $\calM$}, respectively, where
$$
t^n \otimes M_n = \{ t^n \otimes x \mid x \in M_n\} \subseteq R[t, t^{-1}] \otimes_R M
$$
for all $n \in \Bbb Z$. Then $\calR(\calM)$ (resp. $\calR'(\calM)$) is a graded module over $\calR$ (resp. $\calR'$). 

If $F_1 \neq R$, then we define {\it the associated graded ring $\calG$ of $R$ with respect to $\calF$ and the associated graded module $\calG(\calM)$ of $M$ with respect to $\calM$} as follows.
$$
\calG = \calG(\calF) = \calR'/{u\calR'}, \ \  \calG(\calM) = \calR'(\calM)/{u\calR'(\calM)},
$$
where $u = t^{-1}$. Then $\calG(\calM)$ is a graded module over $\calG$ and the composite map
$$
\psi : \calR(\calM) \overset{i}{\longrightarrow} \calR'(\calM) \overset{\varepsilon}{\longrightarrow} \calG(\calM)
$$
is surjective and $\Ker {\psi} = u\calR'(\calM) \cap \calR(\calM) = u[\calR(\calM)]_{+}$, where $[\calR(\calM)]_{+} = \sum_{n>0}t^n \otimes M_n$.

For the rest of this section, we assume that $F_1\neq R $, $\calR = \calR(\calF)$ is Noetherian and $\calR(\calM)$ is finitely generated. 
Then we have the following. The proof of Proposition \ref{2.3} is based on the results \cite[Proposition 5.1]{CGT}. Since it plays an important role in this paper, let us give a brief proof for the sake of completeness. 

\begin{prop}\label{2.3}
The following assertions hold true.
\begin{enumerate}[$(1)$]
\item Let $P \in \Ass_{\calR}\calR(\calM)$. Then $\p \in \Ass_R M$, $P = \p R[t]\cap \calR$ and 
\begin{eqnarray*}
\dim \calR/P = \left\{
 \begin{array}{l}
  \dim R/{\p} + 1 \ \ \ \text{if} \ \dim R/{\p} <\infty,  F_1 \nsubseteq \p, \\
  \dim R/{\p}    \ \ \ \ \ \   \ \ \text{otherwise},
 \end{array}
\right.
\end{eqnarray*}
where  $\p = P \cap R$.
\item $\p R[t] \cap \calR \in \Ass_{\calR} \calR(\calM)$ for every $\p \in \Ass_R M$.
\item Suppose that $M \neq (0)$, $d = \dim_R M < \infty$ and there exists $\p \in \Assh_R M$ such that $F_1 \nsubseteq \p$. Then $\dim_{\calR}\calR(\calM) = d + 1$. 
\end{enumerate}
\end{prop}

\begin{proof}
$(1)$~~ Let $P \in \Ass_{\calR}\calR(\calM)$. Then $P \in \Ass_{\calR}R[t]\otimes_R M$, so that $P = Q \cap \calR$ for some
$$
Q \in \Ass_{R[t]}R[t]\otimes_R M = \bigcup_{\p \in \Ass_R M}\Ass_{R[t]}R[t]/{\p R[t]}.
$$
Thus there exists $\p \in \Ass_R M$ such that $\p = Q \cap R$ and $Q = \p R[t]$. Therefore $P = \p R[t] \cap \calR$, $\p = P \cap R$. Put $\overline{R} = R/{\p}$. Then $\overline{\calF}= \{F_n \overline{R}\}_{n \in \Bbb Z}$ is a filtration of ideals of $\overline{R}$ and $\calR/P \cong \calR(\overline{\calF})$ as graded $R$-algebras. Hence the assertion holds by \cite[Part II, Lemma (2.2)]{GN}.

$(2)$~~Let $\p \in \Ass_R M$. We write $\p = (0):_R x$ for some $x \in M$. Then $(0):_{\calR}\xi = \p R[t] \cap \calR $ where 
$\xi = 1 \otimes x \in [\calR(\calM)]_0$. 

$(3)$~~Follows from the assertions $(1)$, $(2)$.
\end{proof}

\begin{cor}
Suppose that $R$ is a local ring and $M \neq (0)$. Then
\begin{eqnarray*}
\dim_{\calR}\calR(\calM) = \left\{
 \begin{array}{l}
  \dim_R M + 1 \ \ \ \text{if} \ \ \text{there exists }\p \in \Assh_R M \text{ such that }  F_1 \nsubseteq \p, \\
  \dim_R M    \ \ \ \ \ \   \ \ \text{otherwise}.
 \end{array}
\right.
\end{eqnarray*}
\end{cor}

Similarly we are able to determine the structure of associated prime ideals of the extended Rees modules. 

\begin{prop}\label{2.4}
The following assertions hold true.
\begin{enumerate}[$(1)$]
\item Let $P \in \Ass_{\calR'}\calR'(\calM)$. Then $\p \in \Ass_R M$, $P = \p R[t, t^{-1}]\cap \calR'$ and $\dim \calR/P = \dim R/{\p} + 1$, where $\p = P \cap R$.
\item $\p R[t, t^{-1}] \cap \calR' \in \Ass_{\calR'} \calR'(\calM)$ for every $\p \in \Ass_R M$.
\item Suppose that $M \neq (0)$. Then $\dim_{\calR'} \calR'(\calM) = \dim_R M + 1$. 
\end{enumerate}
\end{prop}

Apply Proposition \ref{2.4}, we get the following.

\begin{cor}\label{2.5}
Suppose $R$ is a local ring and $M \neq (0)$. 
Then $\dim_{\calG}\calG(\calM) = \dim_R M$.
\end{cor}

\section{Proof of Theorem \ref{1.2} ~ and Theorem \ref{1.3}}

This section aims to prove Theorem \ref{1.2} and Theorem \ref{1.3}. In what follows, let $(R, \m)$ be a Noetherian local ring, $M \neq (0)$ a finitely generated $R$-module of dimension $d$. Let $\calF = \{F_n\}_{n\in \mathbb{Z}}$ be a filtration of ideals of $R$ with $F_1\neq R$, $\mathcal M = \{M_n\}_{n\in \mathbb{Z}}$ a $\calF$-filtration of $R$-submodules of $M$. We put $\a = \calR(\calF)_{+} = \sum_{n>0}F_nt^n$.

Throughout this section we assume that $\calR=\calR(\calF) $ is a Noetherian ring and $\calR(\calM)$ is finitely generated. Let $1 \le i \le \ell$. We set 
$$
\calD_i = \{M_n\cap D_i\}_{n\in \mathbb Z}, \ \ 
\calC_i = \{[(M_n\cap D_i)+D_{i-1}]/D_{i-1}\}_{n\in \mathbb Z}.
$$
Then $\mathcal D_i$ (resp. $\mathcal C_i$) is a $\calF$-filtration of $R$-submodules of $D_i$ (resp. $C_i$). Look at the following exact sequence
$$
0 \to [\calD_{i-1}]_n \to [\calD_i]_n \to [\calC_i]_n \to 0
$$
of $R$-modules for all $n \in \Bbb Z$. We then have the exact sequences
$$
0 \to \calR(\calD_{i-1}) \to \calR(\calD_i) \to \calR(\calC_i) \to 0
$$
$$
0 \to \calR'(\calD_{i-1}) \to \calR'(\calD_i) \to \calR'(\calC_i) \to 0 \ \ \text{and}
$$
$$
0 \to \calG(\calD_{i-1}) \to \calG(\calD_i) \to \calG(\calC_i) \to 0
$$
of graded modules. Since $\calR(\calD_i)$ is a finitely generated $\calR$-module, so is $\calR(\calC_i)$.

\begin{lem}$($cf. \cite[Proposition 5.1]{CGT}$)$\label{5.9}
$\{\calR'(\calD_i)\}_{0\le i \le \ell}$ is the dimension filtration of $\calR'(\calM)$. If  $F_1 \nsubseteq \p$ for every $\p \in \Ass_R M$, then $\{\calR(\calD_i)\}_{0\le i \le \ell}$ is the dimension filtration of $\calR(\calM)$.   
\end{lem}

\begin{proof}
Let $1 \le i \le \ell$.  Then $\dim_{\calR'}\calR'(\calD_i) = d_i + 1$, since $D_i \neq (0)$. Let $P \in \Ass_{\calR'}\calR'(\calC_i)$. Thanks to Proposition \ref{2.4}, we then have $\dim \calR'/P = d_i + 1 = \dim_{\calR'}\calR'(\calC_i)$. By using \cite[Theorem 2.3]{GHS}, $\{\calR'(\calD_i)\}_{0\le i \le \ell}$ is the dimension filtration of $\calR'(\calM)$. 
Similarly we obtain the last assertion.
\end{proof}

We now ready to prove Theorem \ref{1.2}.

\begin{proof}[Proof of Theorem \ref{1.2}]
The equivalence of conditions $(1)$ and $(2)$ is similar to the proof of Proposition \ref{3.5}. Let us make sure of the last assertion. Look at the following exact sequences
$$
0 \to \calR'(\calC_i) \overset{\varphi}{\to} R[t, t^{-1}]\otimes_R C_i \to X=\Coker{\varphi}  \to 0
$$
of graded $\calR'$-modules for $1 \le i \le \ell$. Since $\calR'(\calC_i)$ is a Cohen-Macaulay $\calR'$-module and $X_u= (0)$, we have $R[t, t^{-1}]\otimes_R C_i$ is Cohen-Macaulay. Therefore $M$ is a sequentially Cohen-Macaulay $R$-module, because $C_i$ is Cohen-Macaulay.
\end{proof}

From now on, we focus our attention on the proof of Theorem \ref{1.3}. To do this, we need some auxiliaries.

\begin{lem}\label{5.1}
Let $P \in \Spec \calR$ such that $P \nsupseteq \a$. If $\calG(\calM)_P \neq (0)$ $($resp. $\calR(\calM)_P \neq (0)$ and $P \supseteq u\a )$, then $\calR(\calM)_P \neq (0)$ $($resp. $\calG(\calM)_P \neq (0) )$. 
When this is the case, the following assertions hold true.
\begin{enumerate}[$(1)$]
\item $\calR(\calM)_P$ is a Cohen-Macaulay $\calR_P$-module if and only if $\calG(\calM)_P$ is a Cohen-Macaulay $\calG_P$-module.
\item $\dim_{\calR_P}\calR(\calM)_P = \dim_{\calR_P}\calG(\calR)_P + 1$. 
\end{enumerate}
\end{lem}

\begin{proof}
Let $P \in \Spec \calR$ such that $P \nsupseteq \a$, but $P \supseteq u\a$. We choose a homogeneous element $\xi = at^n \in \a \setminus P$ where $n > 0$, $a \in F_n$.  
Then we get $x=u\xi  = at^{n-1} \in P$, since $P \supseteq u\a$.
\begin{claim}\label{claim}
If $Q \in \Ass_{\calR}\calR(\calM)$ such that $Q \subseteq P$, then $x \notin Q$. Therefore $x$ is a non-zerodivisor on $\calR(\calM)_P$.
\end{claim}

\begin{proof}[Proof of Claim \ref{claim}]
We assume that there exists $Q \in \Ass_{\calR}\calR(\calM)$ such that $Q \subseteq P$, but $x \in Q$. Write $Q = (0):_{\calR} \eta$ where $\eta = t^{\ell}\otimes m \ (\ell \in \Bbb Z,~m \in M_{\ell})$. Then we have $\xi = at^n \in (0):_{\calR} \eta = Q \subseteq P$, which implies a contradiction.
\end{proof}

Since $P \nsupseteq \a$, we get $\calR_P = \calR'_P$ and $\calR(\calM)_P = \calR'(\calM)_P$. Therefore
$$
(u\a) \calR_P = (u \a)\calR'_P = u \calR'_P = x\calR'_P \ \ \text{and} \ \ (u\a)\calR(\calM) \subseteq u[\calR(\calM)]_{+}.
$$
Hence $[u\calR(\calM)_{+}]_P = x \calR'(\calM)_P = x\calR(\calM)_P$, so that
$$
\calR(\calM)_P/{x\calR(\calM)_P} \cong \calG(\calM)_P
$$
as $\calR_P$-modules.
On the other hand, let $P \in \Spec \calR$ such that $\calG(\calM)_P \neq (0)$. Then $P \supseteq u \a$, since $u \a = u\calR' \cap \calR = \Ker (\calR \overset{i}{\to} \calR' \overset{\varepsilon}{\to} \calG)$. Therefore the assertions immediately come from the above isomorphism. 
\end{proof}

Here we need the following fact, which was originally given by G. Faltings. 

\begin{fact}[\cite{F}]\label{5.3} Let $I$ be an ideal of $R$ and $t \in \Bbb Z$. Consider the following two conditions. 
\begin{enumerate}[$(1)$]
\item There exists an integer $\ell > 0$ such that $I^{\ell}{\cdot}\H^i_{\m}(M) = (0)$ for each $i \neq t$.
\item $M_\p$ is a Cohen Macaulay $R_\p$-module and $t = \dim_{R_\p}M_\p + \dim R/\p$ for every $\p\in\Supp_RM$ but $\p \nsupseteq I$.
\end{enumerate}
Then the implication $(1) \Rightarrow (2)$ holds true. The converse holds, if $R$ is a homomorphic image of a Gorenstein local ring.
\end{fact}

Let $\M$ be a unique graded maximal ideal of $\calR$.
Although a part of the proof of Proposition \ref{5.4} is due to the result \cite{TI}, we note the brief proof for the sake of completeness.

\begin{prop}\label{5.4}
Suppose that $\H^i_\M(\calG(\calM))$ is a finitely graded $\calR$-module for all $i \neq d$. Then $\H^i_\M(\calR(\calM))$ is a finitely graded $\calR$-module for all $i \neq d+1$. 
\end{prop}

\begin{proof}
Passing to the completion and taking the local duality theorem, it is enough to show that there exists an integer $\ell > 0$ such that $\a^{\ell}{\cdot}\H^i_\M(\calR(M)) = (0)$ for every $i \neq d+1$. To see this, let $P \in \Supp_\calR \calR(M)$ such that $P \nsupseteq \a$ and $P \subseteq \M$. Put $L = u \a = u\calR' \cap \calR$.

\begin{claim}\label{claim2}
$\sqrt{P^* + L} \nsupseteq \a$.
\end{claim}

\begin{proof}[Proof of Claim \ref{claim2}]
Suppose that $P^* + L \supseteq \a^{\ell}$ for some $\ell > 0$. Since $\calR/{\a^{\ell}}$ is finitely graded, we can choose an integer $s>0$ such that $[\calR/{\a^{\ell}}]_n = (0)$ for all $n \ge s$. Then
$$
\calR_n = F_nt^n \subseteq [P^*]_n + F_{n+1}t^n
$$
for all $n \ge s$. On the other hand, for each $n \ge 0$, we set 
$$
I_n=\{a \in R \mid at^n \in P^*\}.
$$
Then $I_n$ is an ideal of $R$ and $I_n \subseteq F_n$ and $I_n \supseteq I_{n+1}$ for all $n \ge 0$. 
Hence $F_n \subseteq I_n + F_k$ for all $n\ge s$, $k \in \Bbb Z$.
Since $\calR$ is Noetherian, we get $\calR^{(d)} = R[F_dt^d]$ for some $d>0$, so that $F_{d\ell} = (F_d)^{\ell}$ for all $\ell > 0$.
We then have
$$
F_n \subseteq \bigcap_{\ell >0} [I_n + (F_d)^{\ell}] = I_n
$$
for all $n \ge s$, whence $\calR_n \subseteq P^*$. Thus
$$
\a^s \subseteq \sum_{n \ge s}\calR_n \subseteq P^* \subseteq P.
$$
which is impossible, because $\a \nsubseteq P$.
\end{proof}

 Therefore we can take $Q \in \Min_\calR \calR/[P^* + L]$ such that $\a \nsubseteq Q \subseteq \M$. Then $\calR(\calM)_Q \neq (0)$, because $\calR(\calM)_{P^*} \neq (0)$ and $P^* \subseteq Q$. Thanks to Lemma \ref{5.1}, $\calG(\calM)_Q \neq (0)$. Then  $\calG(\calM)_Q$ is Cohen-Macaulay and $\dim_{\calR_Q}\calG(\calM)_Q + \dim \calR_{\M}/{Q\calR_{\M}} = d$ by using Fact \ref{5.3}. Hence 
 $ \calR(\calM)_Q$ is Cohen-Macaulay and $\dim_{\calR_Q}\calR(\calM)_Q + \dim \calR_{\M}/{Q\calR_{\M}} = d +1$ by Lemma \ref{5.1}.

Since $P^* \subseteq Q$, $\calR(M)_{P^*}$ is Cohen-Macaulay, so is $\calR(M)_P$. We also have
\begin{align*}
d+1 &=  \dim_{\calR_Q}\calR(M)_Q + \dim \calR_{\M}/{Q\calR_{\M}}\\
    &= (\dim_{\calR_{P^*}}\calR(M)_{P^*} + \dim \calR_{Q}/{P^{*}\calR_{Q}} ) + (\dim \calR_{\M}/{P^{*}\calR_{\M}} - \dim \calR_{Q}/{P^{*}\calR_{Q}})\\
    &=  \dim_{\calR_{P^*}}\calR(M)_{P^*} + \dim \calR_{\M}/{P^{*}\calR_{\M}} \\
    &= \dim_{\calR_{P}}\calR(M)_{P} + \dim \calR_{\M}/{P\calR_{\M}}.
\end{align*}

Thanks to Fact \ref{5.3} again, there exists $\ell > 0$ such that
$$
\a^{\ell}{\cdot}\H^i_\M(\calR(\calM)) = (0) \ \ \text{for each} \ \ i \neq d+1 
$$
which shows $\H^i_\M(\calR(\calM))$ is finitely graded.
\end{proof}

We set 
$$
\rma(N) = \max \{n\in \Bbb Z \mid [\H^t_\M(N)]_n \neq (0) \}
$$
for a finitely generated graded $\calR$-module $N$ of dimension $t$, and call it {\it the a-invariant of $N$} (see \cite[DEFINITION (3.1.4)]{GW}). With this notation we have the following.

\begin{lem}\label{5.6}
The following assertions hold true.
\begin{enumerate}[$(1)$]
\item $[\H^{d+1}_{\M}(\calR(\calM))]_n = (0)$ for all $n\ge 0$.
\item If $[\H^{d+1}_{\M}(\calR(\calM))]_{-1} = (0)$, then $\H^{d+1}_{\M}(\calR(\calM))= (0)$.
\end{enumerate}
Consequently $\rma(\calR(\calM)) = -1$, if $\dim_{\calR}\calR(\calM) = d+1$.
\end{lem}

\begin{proof}
We look at the following exact sequences
$$
0 \to L \to \mathcal R(\calM) \to M \to 0
$$
$$
0 \to L(1) \to \mathcal R(\calM) \to \mathcal G(\calM) \to 0
$$
of graded $\calR$-modules, where $L = \calR(\calM)_+$. By applying the local cohomology functors to the above sequences, we get
$$
\H^d_\m(M) \to \H^{d+1}_\M(L) \to \H^{d+1}_\M(\calR(\calM)) \to 0 
$$
and
$$
\H^d_\M(\calG(\calM)) \to \H^{d+1}_\M(L)(1) \to \H^{d+1}_\M(\calR(\calM)) \to 0.
$$
Thus
\begin{align*}
[\H^{d+1}_\M(L)]_n &\cong [\H^{d+1}_\M(\calR(\calM))]_n  \ \ \text{for} \ \ n \neq 0, \ \text{and} \\
[\H^{d+1}_\M(L)]_{n+1} &\to [\H^{d+1}_\M(\calR(\calM))]_n \to 0  \ \ \text{for} \ \ n \in \Bbb Z. 
\end{align*}
Therefore $[\H^{d+1}_\M(\calR(\calM))]_n =(0)$ for $n \ge 0$, because $\H^{d+1}_\M(\calR(\calM))$ is Artinian.  
Moreover we have 
$$
[H^{d+1}_\M(\calR(\calM))]_{-1} \to [\H^{d+1}_\M(\calR(\calM))]_n \to 0
$$
for $n < 0$, so we get the assertion $(2)$.
\end{proof}

We finally arrive at the following Theorem \ref{5.7} which is a module version of the results \cite[Part II, Theorem (1.1)]{GN}, \cite[Theorem 1.1]{V} (see also \cite[Theorem 1.1]{TI}, \cite[Theorem (1.1)]{GS}). 

\begin{thm}\label{5.7} The following conditions are equivalent.
\begin{enumerate}[{\rm (1)}]
\item $\calR(\calM)$ is a Cohen-Macaulay $\calR$-module and $\dim_{\calR}\calR(\calM) = d+1$.
\item $\H^i_\M(\calG(\calM)) = [\H^i_\M(\calG(\calM))]_{-1}$ for every $i < d$ and $\rma(\calG(\calM)) < 0$.
\end{enumerate}
When this is the case, $[\H^i_\M(\calG(\calM))]_{-1}\cong \H^i_{\m}(M)$ as $R$-modules for all $i < d$.
\end{thm}

\begin{proof}
Consider the following exact sequences
$$
(*)~~~~   \cdots \to \H^i_{\m}(L) \to \H^i_\M(\calR(\calM)) \to \H^i_\m(M) \to \H^{i+1}_\M(L) \to \H^{i+1}_\M(\calR(\calM)) \to \cdots
$$
$$
(**)~~~~\cdots \to \H^i_{\m}(L)(1) \to \H^i_\M(\calR(\calM)) \to \H^i_\M(\calG(\calM)) \to \H^{i+1}_\M(L)(1) \to \H^{i+1}_\M(\calR(\calM)) \to \cdots
$$
for each $i < d$. 

Firstly we assume that $\calR(\calM)$ is a Cohen-Macaulay $\calR$-module of dimension $d+1$. Then
$$
\H^i_\m(M) \cong \H^{i+1}_\M(L) \ \ \ \text{and} \ \ \ \H^i_\M(\calG(\calM)) \cong \H^{i+1}_\M(L)(1)
$$
for $i < d$. Therefore we get $\H^i_\M(\calG(\calM)) = [\H^i_\M(\calG(\calM))]_{-1}$ and $[\H^i_\M(\calG(\calM))]_{-1}\cong \H^i_{\m}(M)$ as $R$-modules. Since $\calR(\calM)$ is Cohen-Macaulay, we have
$$
0 \to \H^d_\m(M) \to \H^{d+1}_\M(L) \to \H^{d+1}_\M(\calR(\calM)) \to 0
$$
$$
0 \to \H^d_\M(\calG(\calM)) \to \H^{d+1}_\M(L)(1).
$$
Therefore $\rma(\calG(\calM)) < 0$ by using Lemma \ref{5.6}.

Conversely, let $i<d$. Thanks to the above sequences $(*)$, $(**)$ and our hypothesis, we get
\begin{align*}
[\H^{i+1}_\M(L)]_{n+1} &\cong [\H^{i+1}_\M(\calR(\calM))]_n \\
[\H^{i+1}_\M(L)]_{n+1} &\cong [\H^{i+1}_\M(\calR(\calM))]_{n+1}
\end{align*}
for each $n \ge 0$. Hence $[\H^i_\M(\calR(\calM))]_n = (0)$ for $n \ge 0$, since $\H^{i+1}_\M(\calR(\calM))$ is Artinian.
Moreover, we then have
$$
0 \to [\H^{i+1}_\M(\calR(\calM))]_n \to [\H^{i+1}_\M(\calR(\calM))]_{n-1}. 
$$
for $n < 0$ by above sequences $(*)$ and $(**)$. Thanks to Proposition \ref{5.4}, $\H^{i+1}_\M(\calR(\calM))$ is a finitely graded $\calR$-module for $i < d$.  Whence $[\H^{i+1}_\M(\calR(\calM))]_n = (0)$, which shows $\H^{i+1}_\M(\calR(\calM)) = (0)$ for all $i <d $. Hence $\calR(\calM)$ is a Cohen-Macaulay $\calR$-module of dimension $d+1$.
\end{proof}

\begin{cor}\label{5.8} \label{key} Suppose that $M$ is a Cohen-Macaulay $R$-module. Then the following conditions are equivalent.
\begin{enumerate}[{\rm (1)}]
\item $\calR(\calM)$ is a Cohen-Macaulay $\calR$-module and $\dim_{\calR}\calR(\calM) = d+1$.
\item $\calG(\calM)$ is a Cohen-Macaulay $\calG$-module and $\rma (\calG(\calM))< 0$.
\end{enumerate}
\end{cor}


We now reach the goal of this section.

\begin{proof}[Proof of Theorem \ref{1.3}]
Thanks to Lemma \ref{5.9}, $\calR(\calM)$ is a sequentially Cohen-Macaulay $\calR$-module if and only if $\calR(\calC_i)$ is Cohen-Macaulay for every $1 \le i \le \ell$. The latter condition is equivalent to saying that $\calG(\calC_i)$ is a Cohen-Macaulay $\calG$-module and $\rma(\calG(\calC_i)) < 0$ for all  $1\le i \le \ell$ by Corollary \ref{5.8}. Hence we get the equivalence between $(1)$ and $(2)$.
\end{proof}

We close this section by stating the ring version of Theorem \ref{1.2} and Theorem \ref{1.3}. Let $(R, \m)$ be a Noetherian local ring, $\calF=\{F_n\}_{n \in \Bbb Z}$ a filtration of ideals of $R$ such that $F_1 \neq R$. We assume that $\calR = \calR(\calF)$ is a Noetherian ring. Let $\{D_i\}_{0 \le i \le \ell}$ be the dimension filtration of $R$. Then $\calD_i = \{F_n \cap D_i\}_{n \in \Bbb Z}$ (resp. $\calC_i = \{[F_n \cap D_i + D_{i-1}]/D_{i-1}\}_{n \in \Bbb Z}$) is a $\calF$-filtration of $D_i$ (resp. $C_i$) for all $1 \le i \le \ell$.

\begin{thm}
The following conditions are equivalent.
\begin{enumerate}[$(1)$]
\item $\calR'$ is a sequentially Cohen-Macaulay ring.
\item $\calG$ is a sequentially Cohen-Macaulay ring and $\{\calG(\calD_i)\}_{0 \le i \le \ell}$ is the dimension filtration of $\calG$. 
\end{enumerate}
When this is the case, $R$ is a sequentially Cohen-Macaulay ring.
\end{thm}

\begin{thm}
Suppose that $R$ is a sequentially Cohen-Macaulay ring and $F_1 \nsubseteq \p$ for every $\p \in \Ass R$. Then the following conditions are equivalent.
\begin{enumerate}[{\rm (1)}]
\item $\calR$ is a sequentially Cohen-Macaulay ring. 
\item $\calG$ is a sequentially Cohen-Macaulay ring, $\{\calG (\calD_i)\}_{0 \le i \le \ell}$ is the dimension filtration of $\calG$ and  $\rma(\calG (\calC_i))<0$ for all $1 \le i \le \ell$.
\end{enumerate}
When this is the case, $\calR'$ is a sequentially Cohen-Macaulay ring.
\end{thm}


\section{Sequentially Cohen-Macaulay property in $E^{\natural}$}

In this section let $R=\sum_{n \ge 0}R_n$ be a $\Bbb Z$-graded ring. We put $F_n= \sum_{k \ge n}R_k$ for all $n \in \Bbb Z$. Then $F_n$ is a graded ideal of $R$, $\calF = \{F_n\}_{n\in \Bbb Z}$ is a filtration of ideals of $R$ and $F_1 := R_+ \neq R$. Let $E$ be a graded $R$-module with $E_n = (0)$ for all $n<0$. Put $E_{(n)} = \sum_{k \ge n}E_k$ for all $n \in \Bbb Z$. Then $E_{(n)}$ is a graded $R$-submodule of $E$, $\calE=\{E_{(n)}\}_{n \in \Bbb Z}$ is an $\calF$-filtration of $R$-submodules of $E$ and $E_{(0)} = E$. 
Then we have $R = \calG(\calF)$ and $E = \calG(\calE)$. Set $R^{\natural} := \calR(\calF)$ and $E^{\natural} := \calR(\calE)$.

Suppose that $R$ is a Noetherian ring and $E \neq (0)$ is a finitely generated graded $R$- module with $d = \dim_R E < \infty$. Notice that $R^{\natural}$ is Noetherian and $E^{\natural}$, $\calR'(\calE)$ are finitely generated.

We note the following.

\begin{lem}\label{6.2}
The following assertions hold true.
\begin{enumerate}[$(1)$]
\item $\dim_{\calR'}\calR'(\calE) = d + 1$.
\item Suppose that there exists $\p \in \Assh_R E$ such that $F_1 \nsubseteq \p$. Then $\dim_{R^{\natural}}E^{\natural} = d + 1$.
\end{enumerate}
\end{lem}

\begin{proof}
See Proposition \ref{2.3}, Proposition \ref{2.4}.
\end{proof}


Let $D_0 \subsetneq D_1 \subsetneq \ldots \subsetneq D_{\ell} = E$ be the dimension filtration of $E$. We set $C_i = D_i/{D_{i-1}}$, $d_i = \dim_R D_i$ for every $1\le i\le \ell$. Then $D_i$ is a graded $R$-submodule of $E$ for all $0 \le i \le \ell$.
 Let $1 \le i \le \ell$. Then from the exact sequence
$$
0 \to [D_{i-1}]_{(n)} \to [D_i]_{(n)} \to [C_i]_{(n)} \to 0
$$
of graded $R$-modules for all $n \in \Bbb Z$, we get the exact sequences
$$
0 \to \calR(\calD_{i-1}) \to \calR(\calD_i) \to \calR(\calC_i) \to 0
$$
$$
0 \to \calR'(\calD_{i-1}) \to \calR'(\calD_i) \to \calR'(\calC_i) \to 0 \ \ \text{and}
$$
$$
0 \to \calG(\calD_{i-1}) \to \calG(\calD_i) \to \calG(\calC_i) \to 0
$$
of graded modules, where $\calD_i = \{[D_i]_{(n)} \}_{n\in \mathbb Z}, \ \ 
\calC_i = \{[C_i]_{(n)}\}_{n\in \mathbb Z}$. By the same technique as in the proof of Lemma \ref{5.9}, we obtain the dimension filtration of $\calR'(E)$ and $E^{\natural}$ as follows.

\begin{lem}\label{6.4}
$\{\calR'(\calD_i)\}_{0 \le i \le \ell}$ is the dimension filtration of $\calR'(\calE)$. If $F_1 \nsubseteq \p$ for every $\p \in \Ass_R E$, then $\{\calR(\calD_i)\}_{0 \le i \le \ell}$ is the dimension filtration of $E^{\natural}$. 
\end{lem}

Hence we get the following, which characterize the sequentially Cohen-Macaulayness of $\calR'(\calE)$.

\begin{prop}\label{6.5}
The following conditions are equivalent.
\begin{enumerate}[$(1)$]
\item $\calR'(\calE)$ is a sequentially Cohen-Macaulay $\calR'$-module.
\item $E$ is a sequentially Cohen-Macaulay $R$-module.
\end{enumerate}
\end{prop}

\begin{proof}
$(1) \Rightarrow (2)$~~Follows from the fact that $C_i = \calG(\calC_i)$ for each $1\le i\le \ell$. 

$(2) \Rightarrow (1)$~~We get $\calG(\calC_i)$ is Cohen-Macaulay for all $1\le i \le \ell$. Let $Q \in \Supp_{\calR'}\calR'(\calC_i)$. We may assume $u \notin Q$. Then $\calR'(\calC_i)_u = R[t, t^{-1}]\otimes_R C_i$ is Cohen-Macaulay since $C_i$ is Cohen-Macaulay. Hence $\calR'(\calC_i)_Q$ is a Cohen-Macaulay $\calR'_Q$-module.
\end{proof}

Now we study the question of when $E^{\natural}$ is sequentially Cohen-Macaulay. The key is the following.

\begin{lem}\label{6.6}
Suppose $R_0$ is a local ring, $E$ is a Cohen-Macaulay $R$-module and $F_1 \nsubseteq \p$ for some $\p \in \Assh_R E$. Then the following conditions are equivalent.
\begin{enumerate}[$(1)$]
\item $E^{\natural}$ is a Cohen-Macaulay $R^{\natural}$-module.
\item $\rma(E) < 0$.
\end{enumerate}
\end{lem}

\begin{proof}
Let $P = \m R + R_+$, where $\m$ denotes the maximal ideal of $R_0$. Then $P \supseteq F_1$. Since $R_+(E_{(n)}/{E_{(n+1)}}) = (0)$, $R_+(F_n/{F_{n+1}}) = (0)$ for all $n \in \Bbb Z$, we have 
$$
E = \calG(\calE) \cong \calG(\calE_P), \ \ R = \calG(\calF) \cong \calG(\calF_P).
$$
Suppose that $E^{\natural}$ is a Cohen-Macaulay $R^{\natural}$-module. Then $\calR(\calE_P)$ is Cohen-Macaulay and $\dim_{\calR(R_P)}\calR(\calE_P)= d + 1$, whence $\calG(\calE_P)$ is Cohen-Macaulay and $\rma(\calG(\calE_P))<0$. Therefore we get $\rma(E) < 0$. 

On the other hand, suppose that $\rma(E) < 0$. Then $\calR(\calE_P)$ is a Cohen-Macaulay $\calR(R_P)$-module of dimension $ d+1$. Thus $\calR(\calE)_P$ is Cohen-Macaulay. Now we regard $\calR$ as a ${\Bbb Z}^2$-graded ring with the ${\Bbb Z}^2$-grading as follows:
$$
\calR_{(i,j)} = 
\left\{
\begin{array}{cc}
R_i t^j &    i \ge j \ge 0\\
(0) & otherwise.
\end{array}
\right.
$$
Moreover we set
$$
\calR(\calE)_{(i,j)} = 
\left\{
\begin{array}{cc}
t^j \otimes E_i &    i \ge j \ge 0\\
(0) & otherwise,
\end{array}
\right.
$$
where $t^j \otimes  E_i =\{t^j \otimes x \mid x \in E_i \}$. Then $\calR(\calE)$ is a ${\Bbb Z}^2$-graded $\calR$-module with the above grading $\calR(\calE)_{(i,j)}$. Notice that $\calR_{(0,0)} = R_0$ is a local ring, so that $\calR$ is $H$-local. Let $L$ be the $H$-maximal ideal of the ${\Bbb Z}^2$-graded ring $\calR$. Then we get $P \subseteq L$, whence $L \cap R = P$. Therefore $\calR(\calE)_L$ is a Cohen-Macaulay $\calR_L$-module, so that $E^{\natural}$ is Cohen-Macaulay.
\end{proof}

Our answer is the following.

\begin{thm}\label{6.7}
Suppose that $R_0$ is a local ring, $E$ is a sequentially Cohen-Macaulay $R$-module and $F_1 \nsubseteq \p$ for every $\p \in \Ass_R E$. Then the following conditions are equivalent.
\begin{enumerate}[$(1)$]
\item $E^{\natural}$ is a sequentially Cohen-Macaulay $R^{\natural}$-module.
\item $\rma(C_i) < 0$ for all $1 \le i \le \ell$. 
\end{enumerate}
\end{thm}


\section{Application --Stanley-Reisner algebras--}

In this section, let $V=\{ 1, 2, \ldots, n\}~(n>0)$ be a vertex set, $\Delta$ a simplicial complex on $V$ such that $\Delta \neq \emptyset$. We denote $\calF(\Delta)$ a set of facets of $\Delta$ and  $m=\sharp \calF(\Delta) ~(>0)$ its cardinality. Let $S=k[X_1, X_2, \ldots, X_n]$ a polynomial ring over a field $k$, $R =k[\Delta]= S/{I_{\Delta}}$ the {\it Stanley-Reisner ring of $\Delta$} of dimension $d$, where $I_{\Delta}=(X_{i_1}X_{i_2}\cdots X_{i_r}\mid \{ i_1<i_2<\cdots <i_r \}\notin \Delta)$ the {\it Stanley-Reisner ideal of $R$}. 
 
We consider the Stanley-Reisner ring $R=\sum_{n \ge 0}R_n$ as a $\Bbb Z$-graded ring and put 
$$
I_n= \sum_{k \ge n}R_k = \m^n \ \ \text{for all} \ \ n \in \Bbb Z
$$ 
where $\m = R_+ = \sum_{n>0}R_n$ is a graded maximal ideal of $R$. Then $\calI = \{I_n\}_{n\in \Bbb Z}$ is a $\m$-adic filtration of $R$ and $I_1 := R_+ \neq R$. 

If $\Delta$ is shellable, then $R$ is a sequentially Cohen-Macaulay ring, so by Proposition \ref{6.5} we get the following.

\begin{prop}
If $\Delta$ is shellable, then $\calR'(\m)$ is a sequentially Cohen-Macaulay ring.
\end{prop}

Notice that $\p \nsupseteq I_1$ for every $\p \in \Ass R$ if and only if $F \neq \emptyset$ for all $F \in \calF(\Delta)$, which is equivalent to saying that $\Delta \neq \{\emptyset\}$.

The goal of this section is the following. Here $| F_i |$ denotes the cardinality of $F_i$.

\begin{thm}\label{7.4}
Suppose that $\Delta$ is shellable with shelling order $F_1, F_2, \ldots, F_m \in \calF(\Delta)$ such that $\dim F_1 \ge \dim F_2 \ge \cdots \ge \dim F_m$ 
 and $\Delta \neq \{\emptyset\}$. Then the following conditions are equivalent. 
\begin{enumerate}[$(1)$]
\item $\calR(\m)$ is a sequentially Cohen-Macaulay ring.
\item $m=1$ or if $m \ge 2$, then $| F_i |> \sharp \calF(\Delta_1 \cap \Delta_2)$ for every $2 \le i\le m$, where $\Delta_1 =\left< F_1, F_2, \ldots, F_{i-1}\right>$, $\Delta_2 = \left< F_i\right>$.
\end{enumerate}
\end{thm}

\begin{proof}
Thanks to Theorem \ref{6.7}, $\calR$ is sequentially Cohen-Macaulay if and only if $\rma(C_i) < 0$ for all $1 \le i \le \ell$, where $\{D_i\}_{0 \le i \le \ell}$ is the dimension filtration of $R$, $C_i = D_i/D_{i-1}$ and $d_i = \dim_R D_i$ for all $1 \le i\le \ell$.
If $m=1$, then $R = k[\Delta] \cong k[X_i \mid i \in F_1]$, which is a polynomial ring, so that $\ell = 1$ and $\rma(R) = - | F_1 | < 0$. Hence $\calR$ is a Cohen-Macaulay ring by Lemma \ref{6.6}. 

Suppose that $m > 1$ and the assertion holds for $m-1$. 
We put $\Delta_1=\left< F_1, F_2, \ldots, F_{m-1}\right>$ and $\Delta_2=\left< F_m\right>$. If $\ell = 1$, then $\Delta$ is pure. Look at the following exact sequence
$$
0 \to S/{I_{\Delta}} \to S/{I_{\Delta_1}}\oplus S/{I_{\Delta_2}} \to S/{I_{\Delta_1} + I_{\Delta_2}} \to 0
$$
of graded $R$-modules. We then have 
$$
S/{I_{\Delta_1} + I_{\Delta_2}} \cong k[\Delta_2]/(\overline{\xi})
$$
for some monomials $\xi \in I_{\Delta_1}\setminus I_{\Delta_2}$ in $X_1, X_2, \ldots, X_n$ with $0< \deg \xi = \sharp \calF(\Delta_1 \cap \Delta_2)$. Therefore $\rma (S/{I_{\Delta_1} + I_{\Delta_2}}) = \sharp \calF(\Delta_1 \cap \Delta_2) - | F_m |$. We put $\m= R_+$. Then we have the exact sequence of local cohomology modules as follows
$$
0 \to \H^{d-1}_\m(S/{I_{\Delta_1} + I_{\Delta_2}} ) \to \H^d_\m (S/{I_{\Delta}}) \to \H^d_\m (S/{I_{\Delta_1}}) \oplus \H^d_\m (S/{I_{\Delta_2}}) \to 0.
$$
Thus $\rma(R) = \max\{\sharp \calF(\Delta_1 \cap \Delta_2) - | F_m |, \rma(k[\Delta_1]), \rma(k[\Delta_2])\}$. Hence $\calR$ is sequentially Cohen-Macaulay if and only if $\sharp \calF(\Delta_1 \cap \Delta_2) < | F_m |$ and $\rma(k[\Delta_1])<0$. By using the induction arguments, we get the equivalence between $(1)$ and $(2)$. 

Suppose now that $\ell>1$. Consider the following exact sequence
$$
0 \to I_{\Delta_1}/I_{\Delta} \to S/I_{\Delta} \to S/I_{\Delta_1} \to 0
$$
of graded $R$-modules. Then we have
$$
I_{\Delta_1}/I_{\Delta} \cong {I_{\Delta_1}+I_{\Delta_2}}/I_{\Delta_2} = I_{\Delta_1 \cap \Delta_2}/I_{\Delta_2} = (\overline{\xi})
$$
where $\xi \in I_{\Delta_1}\setminus I_{\Delta_2}$ is a homogeneous element with $0< \deg \xi = \sharp \calF(\Delta_1 \cap \Delta_2)=:t$. Therefore $I_{\Delta_1}/I_{\Delta} \cong S/I_{\Delta_2}(-t)$, so that 
$$
0 \to S/I_{\Delta_2}(-t) \overset{\sigma}{\longrightarrow} S/I_{\Delta} \overset{\varepsilon}{\longrightarrow} S/I_{\Delta_1} \to 0.
$$
We put $L=\Im \sigma$. Then $L \neq (0)$, $\dim_R L = d_1$ and $\rma(L) = t- | F_m |$. We notice here that $L \subseteq D_1$. Now we set ${D_i}' = \varepsilon(D_i)$ for every $1 \le i \le \ell$. Then ${D_1}' \subsetneq {D_2}' \subsetneq \ldots \subsetneq {D_{\ell}}' = k[\Delta_1]$ and ${C_i}' := {D_i}'/{{D_{i-1}}'} \cong C_i$ for all $2 \le i \le \ell$. Hence $\rma(C_i) = \rma({C_i}')$ for $2 \le i \le \ell$.

\vspace{1em}
{\bf \underline{Case 1}} \ \  $L \subsetneq D_1$\ \  (i.e., ${D_1}' \neq (0)$)
\vspace{1em}

In this case ${D_0}':= (0) \subsetneq {D_1}' \subsetneq {D_2}' \subsetneq \ldots \subsetneq {D_{\ell}}' = k[\Delta_1]$ is the dimension filtration of $k[\Delta_1]$. 
Look at the following exact sequence
$$
0 \to L \to D_1 \to D_1' \to 0
$$
of $R$-modules. Then $\rma(D_1) = \max\{\rma(L), \rma({D_1}')\}$.

\vspace{1em}
{\bf \underline{Case 2}} \ \ $L = D_1$  \ \  (i.e., ${D_1}' = (0)$)
\vspace{1em}

Similarly $(0) = {D_1}' \subsetneq {D_2}' \subsetneq \ldots \subsetneq {D_{\ell}}' = k[\Delta_1]$ is the dimension filtration of $k[\Delta_1]$.

Summing up, in any case $\calR$ is a sequentially Cohen-Macaulay ring if and only if $\rma(L) < 0$ and the assertion $(1)$ holds for the ring $k[\Delta_1]$. Hence we get the equivalence of conditions $(1)$ and $(2)$ by using the induction hypothesis.
\end{proof}

\begin{rem}\label{7.2}
If $\Delta$ is shellable, then we can take a shelling order $F_1, F_2, \ldots, F_m \in \calF(\Delta)$ such that $\dim F_1 \ge \dim F_2 \ge \cdots \ge \dim F_m$.
\end{rem}

Apply Theorem \ref{7.4}, we get the following.

\begin{cor}
Under the same notation in Theorem \ref{7.4}. Suppose that $| F_m | \ge 2$. If $\left< F_1, F_2, \ldots, F_{i-1}\right>\cap \left< F_i\right>$ is a simplex for every $2 \le i\le m$, then $\calR(\m)$ is a sequentially Cohen-Macaulay ring.
\end{cor}

Let us give some examples.


\begin{ex}
Let $\Delta = \left< F_1, F_2, F_3\right>$, where $F_1 = \{1,2,3\}$, $F_2=\{2,3, 4\}$ and  $F_3=\{4. 5\}$. Then $\Delta$ is shellable with shelling order $F_1, F_2, F_3 \in \calF(\Delta)$. Then 
$$
\left< F_1\right> \cap \left< F_2\right>, \ \ \left< F_1, F_2\right> \cap \left< F_3\right>
$$ are simplexes, so that $\calR(\m)$ is a sequentially Cohen-Macaulay ring.
$$
{\unitlength 0.1in%
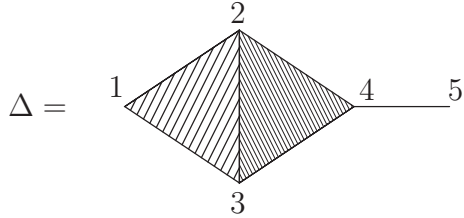
\begin{picture}( 23.1000, 10.0000)( 17.9000,-38.1500)%
%
\special{pn 8}%
\special{pa 2400 3400}%
\special{pa 3000 3000}%
\special{pa 3000 3800}%
\special{pa 2400 3400}%
\special{pa 3000 3000}%
\special{fp}%
%
\special{pn 8}%
\special{pa 2430 3420}%
\special{pa 2460 3360}%
\special{fp}%
\special{pa 2464 3442}%
\special{pa 2528 3315}%
\special{fp}%
\special{pa 2498 3465}%
\special{pa 2595 3270}%
\special{fp}%
\special{pa 2531 3488}%
\special{pa 2662 3225}%
\special{fp}%
\special{pa 2565 3510}%
\special{pa 2730 3180}%
\special{fp}%
\special{pa 2599 3532}%
\special{pa 2798 3135}%
\special{fp}%
\special{pa 2632 3555}%
\special{pa 2865 3090}%
\special{fp}%
\special{pa 2666 3578}%
\special{pa 2932 3045}%
\special{fp}%
\special{pa 2700 3600}%
\special{pa 3000 3000}%
\special{fp}%
\special{pa 2734 3622}%
\special{pa 3000 3090}%
\special{fp}%
\special{pa 2768 3645}%
\special{pa 3000 3180}%
\special{fp}%
\special{pa 2801 3668}%
\special{pa 3000 3270}%
\special{fp}%
\special{pa 2835 3690}%
\special{pa 3000 3360}%
\special{fp}%
\special{pa 2869 3712}%
\special{pa 3000 3450}%
\special{fp}%
\special{pa 2902 3735}%
\special{pa 3000 3540}%
\special{fp}%
\special{pa 2936 3758}%
\special{pa 3000 3630}%
\special{fp}%
\special{pa 2970 3780}%
\special{pa 3000 3720}%
\special{fp}%
%
\special{pn 8}%
\special{pa 3000 3800}%
\special{pa 3600 3400}%
\special{pa 3000 3000}%
\special{pa 3000 3800}%
\special{pa 3600 3400}%
\special{fp}%
%
\special{pn 4}%
\special{pa 3002 3799}%
\special{pa 3000 3795}%
\special{fp}%
\special{pa 3019 3788}%
\special{pa 3000 3750}%
\special{fp}%
\special{pa 3036 3776}%
\special{pa 3000 3705}%
\special{fp}%
\special{pa 3052 3765}%
\special{pa 3000 3660}%
\special{fp}%
\special{pa 3069 3754}%
\special{pa 3000 3615}%
\special{fp}%
\special{pa 3086 3742}%
\special{pa 3000 3570}%
\special{fp}%
\special{pa 3103 3731}%
\special{pa 3000 3525}%
\special{fp}%
\special{pa 3120 3720}%
\special{pa 3000 3480}%
\special{fp}%
\special{pa 3137 3709}%
\special{pa 3000 3435}%
\special{fp}%
\special{pa 3154 3698}%
\special{pa 3000 3390}%
\special{fp}%
\special{pa 3171 3686}%
\special{pa 3000 3345}%
\special{fp}%
\special{pa 3188 3675}%
\special{pa 3000 3300}%
\special{fp}%
\special{pa 3204 3664}%
\special{pa 3000 3255}%
\special{fp}%
\special{pa 3221 3652}%
\special{pa 3000 3210}%
\special{fp}%
\special{pa 3238 3641}%
\special{pa 3000 3165}%
\special{fp}%
\special{pa 3255 3630}%
\special{pa 3000 3120}%
\special{fp}%
\special{pa 3272 3619}%
\special{pa 3000 3075}%
\special{fp}%
\special{pa 3289 3608}%
\special{pa 3000 3030}%
\special{fp}%
\special{pa 3306 3596}%
\special{pa 3011 3008}%
\special{fp}%
\special{pa 3322 3585}%
\special{pa 3045 3030}%
\special{fp}%
\special{pa 3339 3574}%
\special{pa 3079 3052}%
\special{fp}%
\special{pa 3356 3562}%
\special{pa 3112 3075}%
\special{fp}%
\special{pa 3373 3551}%
\special{pa 3146 3098}%
\special{fp}%
\special{pa 3390 3540}%
\special{pa 3180 3120}%
\special{fp}%
\special{pa 3407 3529}%
\special{pa 3214 3142}%
\special{fp}%
\special{pa 3424 3518}%
\special{pa 3248 3165}%
\special{fp}%
\special{pa 3441 3506}%
\special{pa 3281 3188}%
\special{fp}%
\special{pa 3458 3495}%
\special{pa 3315 3210}%
\special{fp}%
\special{pa 3474 3484}%
\special{pa 3349 3232}%
\special{fp}%
\special{pa 3491 3472}%
\special{pa 3382 3255}%
\special{fp}%
\special{pa 3508 3461}%
\special{pa 3416 3278}%
\special{fp}%
\special{pa 3525 3450}%
\special{pa 3450 3300}%
\special{fp}%
\special{pa 3542 3439}%
\special{pa 3484 3322}%
\special{fp}%
\special{pa 3559 3428}%
\special{pa 3518 3345}%
\special{fp}%
\special{pa 3576 3416}%
\special{pa 3551 3368}%
\special{fp}%
\special{pa 3592 3405}%
\special{pa 3585 3390}%
\special{fp}%
%
\special{pn 8}%
\special{pa 3600 3400}%
\special{pa 4100 3400}%
\special{fp}%
\put(17.9000,-34.5000){\makebox(0,0)[lb]{$\Delta=$}}%
\put(23.1500,-33.5000){\makebox(0,0)[lb]{1}}%
\put(36.2500,-33.6500){\makebox(0,0)[lb]{4}}%
\put(40.9000,-33.6500){\makebox(0,0)[lb]{5}}%
\put(29.5000,-29.6000){\makebox(0,0)[lb]{2}}%
\put(29.5000,-39.6000){\makebox(0,0)[lb]{3}}%
\end{picture}}%

$$
\vspace{0.01cm}
\end{ex}

\begin{ex}
Let $\Delta = \left< F_1, F_2, F_3, F_4\right>$, where $F_1 = \{1,2,5\}$, $F_2=\{2,3\}$, $F_3=\{3,4\}$ and $F_4=\{4,5\}$. Notice that $\Delta$ is a shellable simplicial complex with shelling order $F_1, F_2, F_3, F_4 \in \calF(\Delta)$. We put $\Delta_1 = \left< F_1, F_2, F_3\right>$ and  $\Delta_2 = \left< F_4\right>$.  Then 
$$
\sharp \calF(\Delta_1 \cap \Delta_2) = 2 = | F_4 |,
$$
whence $\calR(\m)$ is not sequentially Cohen-Macaulay.
$$
{\unitlength 0.1in%
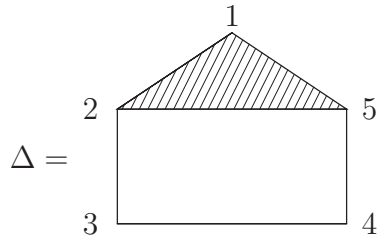
\begin{picture}( 19.5500, 11.5500)( 27.3500,-50.0000)%
%
\special{pn 8}%
\special{pa 3400 5000}%
\special{pa 4600 5000}%
\special{pa 4600 4400}%
\special{pa 3400 4400}%
\special{pa 3400 5000}%
\special{pa 4600 5000}%
\special{fp}%
%
\special{pn 8}%
\special{pa 3400 4400}%
\special{pa 4000 4000}%
\special{pa 4600 4400}%
\special{pa 3400 4400}%
\special{pa 4000 4000}%
\special{fp}%
%
\special{pn 4}%
\special{pa 3425 4400}%
\special{pa 3438 4375}%
\special{fp}%
\special{pa 3470 4400}%
\special{pa 3505 4330}%
\special{fp}%
\special{pa 3515 4400}%
\special{pa 3572 4285}%
\special{fp}%
\special{pa 3560 4400}%
\special{pa 3640 4240}%
\special{fp}%
\special{pa 3605 4400}%
\special{pa 3708 4195}%
\special{fp}%
\special{pa 3650 4400}%
\special{pa 3775 4150}%
\special{fp}%
\special{pa 3695 4400}%
\special{pa 3842 4105}%
\special{fp}%
\special{pa 3740 4400}%
\special{pa 3910 4060}%
\special{fp}%
\special{pa 3785 4400}%
\special{pa 3978 4015}%
\special{fp}%
\special{pa 3830 4400}%
\special{pa 4022 4015}%
\special{fp}%
\special{pa 3875 4400}%
\special{pa 4056 4038}%
\special{fp}%
\special{pa 3920 4400}%
\special{pa 4090 4060}%
\special{fp}%
\special{pa 3965 4400}%
\special{pa 4124 4082}%
\special{fp}%
\special{pa 4010 4400}%
\special{pa 4158 4105}%
\special{fp}%
\special{pa 4055 4400}%
\special{pa 4191 4128}%
\special{fp}%
\special{pa 4100 4400}%
\special{pa 4225 4150}%
\special{fp}%
\special{pa 4145 4400}%
\special{pa 4259 4172}%
\special{fp}%
\special{pa 4190 4400}%
\special{pa 4292 4195}%
\special{fp}%
\special{pa 4235 4400}%
\special{pa 4326 4218}%
\special{fp}%
\special{pa 4280 4400}%
\special{pa 4360 4240}%
\special{fp}%
\special{pa 4325 4400}%
\special{pa 4394 4262}%
\special{fp}%
\special{pa 4370 4400}%
\special{pa 4428 4285}%
\special{fp}%
\special{pa 4415 4400}%
\special{pa 4461 4308}%
\special{fp}%
\special{pa 4460 4400}%
\special{pa 4495 4330}%
\special{fp}%
\special{pa 4505 4400}%
\special{pa 4529 4352}%
\special{fp}%
\special{pa 4550 4400}%
\special{pa 4562 4375}%
\special{fp}%
\special{pa 4596 4398}%
\special{pa 4595 4400}%
\special{fp}%
\put(40.0000,-39.1000){\makebox(0,0){1}}%
\put(32.6000,-44.0000){\makebox(0,0){2}}%
\put(32.6000,-50.0000){\makebox(0,0){3}}%
\put(47.2000,-50.0000){\makebox(0,0){4}}%
\put(47.2000,-44.0000){\makebox(0,0){5}}%
\put(29.9000,-46.5000){\makebox(0,0){$\Delta=$}}%
\end{picture}}%

$$
\end{ex}

\vspace{0.5cm}

\begin{ac}
The authors would like to thank Professor Shiro Goto for his valuable advice and comments.
\end{ac}





\begin{thebibliography}{GHS}

\bibitem[BR]{BR}
{\sc J. W. Brewer and E. A. Rutter}, \textit{Must $R$ be Noetherian if $R^G$ is Noetherian}, Comm. Algebra, {\bf 5} (1977), 969-979.

\bibitem[CC]{CC}
{\sc N. T. Cuong and D. T. Cuong}, {\em On sequentially Cohen-Macaulay modules,} Kodai Math. J., {\bf (30)}  (2007), 409-428.

\bibitem[CGT]{CGT} 
{\sc N. T. Cuong, S. Goto and H. L. Truong}, {\em The equality $I^2=\q I$ in sequentially Cohen-Macaulay rings}, J. Algebra, {\bf (379)}  (2013), 50-79.




\bibitem[F]{F} 
{\sc G. Faltings}, \textit{\"{U}ber die Annulatoren lokaler Kohomologiegruppen}, Archiv der Math., {\bf 30} (1978), 473--476.


\bibitem[GHS]{GHS}
{\sc S. Goto, Y. Horiuchi and H. Sakurai}, \textit{Sequentially Cohen-Macaulayness versus parametric decomposition of powers of parameter ideals}, J. Comm. Algebra, {\bf 2} (2010), 37--54.

\bibitem[GN]{GN} 
{\sc S. Goto and K. Nishida}, \textit{The Cohen-Macaulay and Gorenstein properties of Rees algebras associated to fltrations}, Mem. Amer. Math. Soc., {\bf 110} (1994).

\bibitem[GS]{GS}
{\sc S. Goto and Y. Shimoda}, \textit{On the Rees algebra of Cohen-Macaulay local rings}, Commutative Algebra, Lecture Note in Pure and Applied Mathematics, Marcel Dekker Inc, {\bf 68} (1982), 201-231.


\bibitem[GW]{GW} 
{\sc S. Goto and K. Watanabe}, {\it On graded rings, I}, J. Math. Soc. Japan, {\bf 30} (1978), 179--213.


\bibitem[Sch]{Sch}
{\sc P. Schenzel}, \textit{On the dimension filtration and Cohen-Macaulay filtered modules},  in: Proc. of the Ferrara Meeting in honour of Mario Fiorentini, University of Antwerp, Wilrijk, Belgium, (1998) pp. 245-264.

\bibitem[St]{St} 
{\sc R. P. Stanley}, Combinatorics and commutative algebra, Second Edition, Birkh{\" a}user, Boston, 1996.


\bibitem[TI]{TI}  
{\sc N. Trung and S. Ikeda}, \textit{When is the Rees algebra Cohen-Macaulay?}, Comm. Algebra {\bf 17} (1989), 2893-2922.


\bibitem[V]{V}  
{\sc D. Q. Viet}, \textit{A note on the Cohen-Macaulayness of Rees Algebra of filtrations}, Comm. Algebra {\bf 21} (1993), 221-229.




\end{thebibliography}
\end{document}